\documentclass[11pt]{article}
\usepackage{amsfonts}
\usepackage{mathrsfs}

\usepackage[latin1]{inputenc}
\usepackage{amsmath,amssymb}
\usepackage{latexsym}
\usepackage[active]{srcltx}
\usepackage[
bookmarks=true,         
bookmarksnumbered=true, 
colorlinks=true, pdfstartview=FitV, linkcolor=blue, citecolor=blue,
urlcolor=blue]{hyperref}

\newtheorem{theorem}{Theorem}[section]

\newtheorem{lemma}[theorem]{Lemma}

\newtheorem{remark}[theorem]{Remark}
\numberwithin{equation}{section}
\def\R{{\mathbb R}}
\def\E{{{\mathbb E}\,}}

\def\N{{\mathbb N}}

\def\Var{{\mathop {{\rm Var\, }}}}

\def\square{{\vcenter{\vbox{\hrule height.3pt
        \hbox{\vrule width.3pt height5pt \kern5pt
           \vrule width.3pt}
        \hrule height.3pt}}}}

\def\tlint{{- \kern-0.85em \int \kern-0.2em}}
\def\dlint{{- \kern-1.05em \int \kern-0.4em}}

\newenvironment{proof}[1][Proof]{\noindent\textit{#1.} }{\hfill \rule{0.5em}{0.5em}}

\begin{document}

\title{Derivatives of local times for some Gaussian fields}
\date{\today}

\author{Minhao Hong and Fangjun Xu\thanks{M. Hong is partially supported by National Natural Science Foundation of China (Grant No.11871219) and ECNU Academic Innovation Promotion Program for Excellent Doctoral Students (YBNLTS2019-010). F. Xu is partially supported by National Natural Science Foundation of China (Grant No.11871219, No.11871220).} \\
}

\maketitle

\begin{abstract}
\noindent  In this article, we consider derivatives of local time for a $(2,d)$-Gaussian field 
\[
Z=\big\{ Z(t,s)= X^{H_1}_t -\widetilde{X}^{H_2}_s, s,t \ge 0\big\},
\] 
where $X^{H_1}$ and $\widetilde{X}^{H_2}$ are two independent processes from a class of $d$-dimensional centered Gaussian processes satisfying certain local nondeterminism property.  We first give a condition for existence of derivatives of the local time. Then, under this condition, we show that derivatives of the local time are H\"{o}lder continuous in both time and space variables. Moreover, under some additional assumptions, we show that this condition is also necessary for existence of derivatives of the local time at the origin.  


\end{abstract}

\section{Introduction}

For $H\in (0,1)$,  let $X^H= \{ X^H_t:\, t\geq 0\} $ be a $d$-dimensional centered Gaussian stochastic process whose components $X^{H,\ell} (1\le \ell \le d)$ are independent and identically distributed, and satisfy the following local nondeterminism property: for any $T>0$, there exists a positive constant $\kappa_{T,m,H}$ depending only on $T$, $m$ and $H$, such that for any $0=t_0<t_1 <\dots < t_m<2T$ and $x_i \in \mathbb{R} (1\le i \le m)$, we have
\begin{equation} \label{ln}
\Var \Big( \sum_{i=1}^m x_i   (X^{H,\ell}_{t_i} -X^{H,\ell}_{t_{i-1}}) \Big)\geq \kappa_{T,m,H}  \sum_{i=1}^m x_i^2 (t_i -t_{i-1})^{2H}.
\end{equation}
Let $G^d_L$ be the class of all such $d$-dimensional centered Gaussian processes $X^H$ and $G^d_{L,U}$ the class of $X^H\in G^d_L$ posessing the following property:  there is a positive constant $C_{T,H}$ depending only on $T$ and $H$ such that $\Var(X^{H,1}_t)\leq C_{T,H}|t|^{2H}$ for all $t\in[0,2T]$.

From results in \cite{sxy}, we can easily see that the $d$-dimensional Gaussian processes given below are in $G^d_{L,U}$:

\noindent
(i) {\it Bifractional Brownian motion (bi-fBm)}. The covariance function for components of this process is given by
\[
\E(X^{H,\ell}_t X^{H,\ell}_s)=2^{-K_0}\big[(t^{2H_0}+s^{2H_0})^{K_0}-|t-s|^{2H_0K_0}\big],
\]
where $H_0\in(0,1)$ and $K_0\in(0,1]$.  Here $H=H_0K_0$ and $K=1$ gives the classic fractional Brownian motion (fBm) case with Hurst parameter $H=H_0$.
 
\noindent
(ii) {\it Subfractional Brownian motion (sub-fBm)}.  The covariance function for components of this process  is given by
\[
\E(X^{H,\ell}_t X^{H,\ell}_s)=t^{2H}+s^{2H}-\frac{1}{2}\big[(t+s)^{2H}+|t-s|^{2H}\big],
\]
where $H\in(0,1)$.  
 
It is well-known that the Dirac function $\delta(x)$ on $\R^d$ can be approximated by 
\[
p_{\varepsilon}(x)=\frac{1}{(2\pi \varepsilon)^{\frac{d}{2}}} e^{-\frac{|x|^2}{2\varepsilon}}=\frac{1}{(2\pi)^d}\int_{\R^d} e^{\iota y\cdot x}e^{-\frac{\varepsilon|y|^2}{2}}\, dy.
\]
So, for multi-index $\mathbf{k}=(k_1,\cdots, k_d)$ with all $k_i$ being nonnegative integers, we can approximate 
\[
\delta^{(\mathbf{k})}=\frac{\partial^{\bf k}}{\partial x^{k_1}_1\cdots \partial x^{k_d}_d} \delta(x)
\]
by 
\[
p^{(\mathbf{k})}_{\varepsilon}(x)=\frac{\partial^\mathbf{k}}{\partial x^{k_1}_1\cdots \partial x^{k_d}_d}p_{\varepsilon}(x) =\frac{\iota^{|\mathbf{k}|}}{(2\pi)^d}\int_{\R^d} \Big(\prod^d_{i=1}y^{k_i}_i\Big)\, e^{\iota y\cdot x}e^{-\frac{\varepsilon|y|^2}{2}}\, dy,
\]
where $|\mathbf{k}|=\sum\limits^d_{i=1}k_i$.

Let $X^{H_1}$ and $\widetilde{X}^{H_2}$ be two independent Gaussian processes in $G^d_L$ with parameters $H_1$ and  $H_2$, respectively. Then 
\begin{align}
Z=\big\{Z(t,s)= X^{H_1}_t -\widetilde{X}^{H_2}_s, s,t \ge 0\big\}  \label{gf}
\end{align}
is a $(2,d)$-Gaussian field. For any $T>0$ and $x\in\R^d$, if
\begin{align}\label{epsilon}
L^{(\mathbf{k})}_{\varepsilon}(T,x):=\int^T_0\int^T_0 p^{(\mathbf{k})}_{\varepsilon}(X^{H_1}_t-\widetilde{X}^{H_2}_s-x)\, ds\, dt
\end{align}
converges to a random variable in $L^p (p\geq 1)$  when $\varepsilon\downarrow 0$, we denote the limit by $L^{(\mathbf{k})}(T,x)$ and call it  the $\mathbf{k}$-th derivatives of local time for the $(2,d)$-Gaussian field $Z$.  If it exists,  $L^{(\mathbf{k})}(T,x)$ admits the following $L^p$-representation 
\begin{align} \label{dlt}
L^{(\mathbf{k})}(T,x)=\int^T_0\int^T_0 \delta^{({\bf k})}(X^{H_1}_t-\widetilde{X}^{H_2}_s+x)\, ds\, dt.
\end{align}

When $\mathbf{k}=\mathbf{0}$, $L^{(\mathbf{0})}(T,x)$ is just the local time of the $(2,d)$-Gaussian field $Z$ at $x$. Local time of Gaussian processes or Gaussian fields are important subjects in probability theory and has a long history, see, e.g., \cite{Be, p, gh, ghr} and references therein. When $x=0$, $L^{(\mathbf{0})}(T,0)$ is also called the intersection local time of $X^{H_1}$ and  $\widetilde{X}^{H_2}$. Recently, intersection local time for independent fBms  and its derivatives have received a lot of attention. If $X^{H_1}$ and  $\widetilde{X}^{H_2}$ are two independent $d$-dimensional fBms with the same Hurst parameter $H$,  Nualart and Ortiz-Latorre in \cite{no} proved that $L^{(\mathbf{0})}(T,0)$ exists in $L^2$ if and only if $Hd<2$. This result was later extended to $(N,d)$-fBms with Hurst parameters $H_1$ and $H_2$ by  Wu and Xiao in \cite{wx}, where they also established the regularity of the corresponding intersection local time. When $\mathbf{k}\neq \mathbf{0}$, Yan in \cite{y} showed that $L^{(1)}_{\varepsilon}(T,x)$ converges in $L^p$ $(p>0)$ if $\frac{1}{H_1}+\frac{1}{H_2}>3$ for the $1$-dimensional fBm case; Guo, Hu and Xiao in \cite{ghx} gave a condition for existence of $L^{(\mathbf{k})}(T,0)$ and showed the exponential integrability of $L^{(\mathbf{k})}(T,0)$ for the $d$-dimensional fBm case. When the $(2,d)$-Gaussian field $Z$ is replaced by a $1$-dimensional fBm, Jaramillo, Nourdin and Peccati in \cite{jnp} gave a sharp condition for existence (in $L^2$) of derivatives of the local time and established their H\"{o}lder continuity in the time variable. 

In this paper we consider existence of $L^{(\mathbf{k})}(T,x)$ for independent processes from a large class of $d$-dimensional centered Gaussian processes including fBms, bi-fBms and sub-fBms. In addition, we do not require that $X^{H_1}$ and $\widetilde{X}^{H_2}$ are both fBms.  They can be different, for example, $X^{H_1}$ can be a bi-fBm while $\widetilde{X}^{H_2}$  a sub-fBm.  We give a mild condition for existence of $L^{(\mathbf{k})}(T,x)$ and then establish its H\"{o}lder continuity in both time and space variables. Our condition is sharp since it is also necessary for existence of $L^{(\mathbf{k})}(T,0)$ if $X^{H_1}$ and $\widetilde{X}^{H_2}$ satisfy certain additional property, which is posed by Gaussian processes, say fBms, bi-fBms and sub-fBms. A reason for focusing on two independent Gaussian processes is that the methodology developed here can be easily used to obtain the corresponding results for one Gaussian process or $k$ $(k\geq 3)$ independent Gaussian processes. Moreover, this paper can be viewed as an extension of \cite{sxy, nx2} where central limit theorems for functionals of $X^{H}_t -\widetilde{X}^{H}_s$ are not available for $H\leq \frac{2}{d+2}$, see \cite{jnp} for this phenomenon in the one $1$-dimensional fBm case. Here the main difficulty comes from the second independent Gaussian process. Especially in the proof of Theorem \ref{thm2},  we need some kind of chaining argument to get the main ingredient in $\E[ |L^{(\mathbf{k})}_{\varepsilon}(T,0)|^2]$ as $\varepsilon\downarrow 0$.

The following are main results of this paper.
\begin{theorem}\label{thm1}
Assume that $X^{H_1}=\{X^{H_1}_t:\, t\geq 0\}$ and $\widetilde{X}^{H_2}=\{\widetilde{X}^{H_2}_t:\, t\geq 0\}$ are two independent Gaussian processes in $G^d_L$ with parameters $H_1, H_2\in(0,1)$, respectively. If $\frac{H_1H_2}{H_1+H_2}(2|\mathbf{k}|+d)<1$, then the $\mathbf{k}$-th derivative of local time $L^{(\mathbf{k})}(T,x)$ exists in $L^p$ for any $p\in[1,\infty)$. Moreover, $L^{(\mathbf{k})}(T,x)$ has a modification which is $\theta_1$-H\"{o}lder continuous in space for all $\theta_1\in(0, 1\wedge (\frac{1}{H_1}+\frac{1}{H_2}-2|\mathbf{k}|-d))$ and $\theta_2$-H\"{o}lder continuous in time for all $\theta_2\in(0, 1-\frac{H_1H_2}{H_1+H_2}(|\mathbf{k}|+d))$.

\end{theorem}

\begin{theorem}\label{thm2}
Assume that $X^{H_1}=\{X^{H_1}_t:\, t\geq 0\}$ and $\widetilde{X}^{H_2}=\{\widetilde{X}^{H_2}_t:\, t\geq 0\}$ are two independent Gaussian processes in $G^d_{L,U}$ with parameters $H_1, H_2\in(0,1)$, respectively. The following statements are true: 
\begin{itemize}
\item[(i)] if $\frac{H_1H_2}{H_1+H_2}d\geq 1$, then $L^{(\mathbf{0})}_{\varepsilon}(T,0)$ diverges in $L^2$ as $\varepsilon\downarrow 0$; 
\item[(ii)] if $\frac{H_1H_2}{H_1+H_2}d\leq 1$ and $\frac{H_1H_2}{H_1+H_2}(2|\mathbf{k}|+d)\geq 1$, then $L^{(\mathbf{k})}_{\varepsilon}(T,0)$ diverges in $L^2$ as $\varepsilon\downarrow 0$.
\end{itemize}
\end{theorem}

\begin{remark}  For independent $d$-dimensional centered Gaussian processes $X^{H_1}$ and $\widetilde{X}^{H_2}$ in $G^d_{L,U}$,  Theorem \ref{thm1} and statement (i) in Theorem \ref{thm2} imply that $L^{(\mathbf{0})}(T,0)$ exists in $L^2$ if and only if $\frac{H_1H_2}{H_1+H_2}d<1$. Theorem \ref{thm1} and statement (ii) in Theorem \ref{thm2} say that, if $L^{(\mathbf{0})}(T,0)$ exists in $L^2$, then $L^{(\mathbf{k})}(T,0) (\mathbf{k}\neq \mathbf{0})$ exists in $L^2$ if and only if $\frac{H_1H_2}{H_1+H_2}(2|\mathbf{k}|+d)<1$.  In the proof of Theorem \ref{thm2} below, we also give concrete divergence rates of $\E |L^{(\mathbf{k})}_{\varepsilon}(T,0)|^2$ as $\varepsilon\downarrow 0$ in different cases.  The asymptotic behavior of $L^{(\mathbf{k})}_{\varepsilon}(T,0)$ as $\varepsilon\downarrow 0$ in these cases will be studied in a future paper.
\end{remark}

\begin{remark}
For $N\in\N$, define the $(N, d)$-Gaussian field 
\[
 Z^{N}=\Big\{\sum^{N}_{j=1} X^{j,H_j}_{t_j}:\; t_j\geq 0,\, j=1,\dots,d\Big\},
\]
where $X^{j,H_j}_{t_j}$ are independent Gaussian processes in $G^d_{L}$. Replace $Z$ in $L^{(\mathbf{k})}(T,x)$ and $L^{(\mathbf{k})}(T,x)$ by $Z^{N}$ and denote the new terms as $L^{(\mathbf{k})}_{N}(T,x)$ and $L^{(\mathbf{k})}_{N, \varepsilon}(T,x)$, respectively. Using similar arguments as in the proof of Theorems \ref{thm1} and \ref{thm2},  we can also obtain that
\begin{enumerate}
\item[(1)] If $2|\mathbf{k}|+d<\sum\limits^d_{j=1}H^{-1}_j$, $L^{(\mathbf{k})}_{N}(T,x)$ exists in $L^p$ for any $p\in[1,\infty)$. Moreover, $L^{(\mathbf{k})}_{N}(T,x)$ has a modification which is $\theta_1$-H\"{o}lder continuous in space for all $\theta_1\in(0, 1\wedge (\sum\limits^d_{j=1}\frac{1}{H_j}-2|\mathbf{k}|-d))$ and $\theta_2$-H\"{o}lder continuous in time for all $\theta_2\in(0, 1-\frac{|\mathbf{k}|+d}{\sum^d_{j=1}H^{-1}_j})$.
\item[(2)]  If $d\geq \sum\limits^d_{j=1}H^{-1}_j$ and $X^{j,H_j}_{t_j}\in G^d_{L,U}$ for each $j=1,2\dots,\ell$, then $L^{(\mathbf{0})}_{N,\varepsilon}(T,0)$ diverges in $L^2$ as $\varepsilon\downarrow 0$.
\item[(3)]  If $(2|\mathbf{k}|+d)\geq \sum\limits^d_{j=1}H^{-1}_j\geq d$ and $X^{j,H_j}_{t_j}\in G^d_{L,U}$ for each $j=1,2\dots,\ell$, then $L^{(\mathbf{k})}_{N,\varepsilon}(T,0)$ diverges in $L^2$ as $\varepsilon\downarrow 0$.
\end{enumerate}

\end{remark}

After some preliminaries in Section 2, Section 3 is devoted to the proofs of Theorems \ref{thm1} and \ref{thm2}. Throughout this paper, if not mentioned otherwise, the letter $c$, 
with or without a subscript, denotes a generic positive finite constant whose exact value is independent of $m$ and may change from line to line.  For any $x,y\in\R^d$, we use $x\cdot y$ to denote the usual inner product  and $|x|=(\sum\limits^d_{i=1}|x_i|^2)^{1/2}$.  Moreover, we use $\iota$ to denote $\sqrt{-1}$.

\bigskip

\section{Preliminaries}

In this section, we give two lemmas for Gaussian processes $X^H$ in $G^d_{L}$. Lemma \ref{lma1} is needed in the proof of Theorem \ref{thm1}, while Lemma \ref{lma2} plays an important role in the proof of Theorem \ref{thm2}.

\begin{lemma} \label{lma1}  For $0<s_1<s_2<\cdots<s_m<2T$, $k\in \N\cup\{0\}$ and $\varepsilon\geq 0$, there exists a constant $C_{T, k, m,H}$ depending only  on $T$, $k$, $m$ and $H$ such that 
\begin{align*}
&\int_{\R^{md}} \exp\bigg(-\frac{1}{2}\Var\Big(\sum\limits^m_{j=1} y_j \cdot X^{H}_{s_j}\Big)-\frac{\varepsilon}{2}\sum\limits^m_{j=1}|y_j|^2\bigg) \prod^m_{j=1} |y_{j}|^{k}\, dy\\
&\qquad\qquad \leq  C_{T, k, m,H}  \sum_{\mathcal{S}} \prod^{m}_{j=1}  \Big[(s_j-s_{j-1})^{2H}+\varepsilon\Big]^{-\frac{1+k(p_j+\overline{p}_{j-1})}{2}},
\end{align*}
where $\mathcal{S}=\big\{p_i, \overline{p}_i:  p_i\in \{0,1\}, p_i+\overline{p}_i=1,\, i=1,\dots,m-1, p_m=1\big\}$.
\end{lemma}
\begin{proof}
Make the change of variables $x_j=\sum\limits^m_{k=j} y_k$ for $j=1,2,\cdots,m$ with the convention $x_{m+1}=0$. Then, using the local nondeterminism property (\ref{ln}) and Lemma \ref{a4}, we can obtain
\begin{align*}
&\int_{\R^{md}}  \exp\bigg(-\frac{1}{2}\Var\Big(\sum\limits^m_{j=1} y_j \cdot X^{H}_{s_j}\Big)-\frac{\varepsilon}{2}\sum\limits^m_{j=1}|y_j|^2\bigg)\prod^m_{j=1} |y_{j}|^{k}\, dy\\
&\leq \int_{\R^{md}} \exp\bigg(-\frac{\kappa_{T, m,H} }{2}\sum\limits^m_{j=1} |x_j|^2\Big[(s_j-s_{j-1})^{2H}+\frac{2\varepsilon}{m(m+1)\kappa_{T, m,H}}\Big]\bigg)\prod^m_{j=1} |x_j-x_{j+1}|^{k}\, dx.
\end{align*}

Let $c=\kappa_{T, m,H} \wedge \frac{2}{m(m+1)}$. Then
\begin{align*}
&\int_{\R^{md}}  \exp\bigg(-\frac{1}{2}\Var\Big(\sum\limits^m_{j=1} y_j \cdot X^{H}_{s_j}\Big)-\frac{\varepsilon}{2}\sum\limits^m_{j=1}|y_j|^2\bigg)\prod^m_{j=1} |y_{j}|^{k}\, dy\\
&\leq  2^{km}\int_{\R^{md}} \exp\bigg(-\frac{c}{2} \sum\limits^m_{j=1} |x_j|^2\big[(s_j-s_{j-1})^{2H}+\varepsilon\big]\bigg) \prod^{m}_{j=1} (|x_j|^k+|x_{j+1}|^{k})\, dx\\
&= 2^{km} \sum_{\mathcal{S}}  \int_{\R^{md}} \exp\bigg(-\frac{c}{2} \sum\limits^m_{j=1} |x_j|^2\big[(s_j-s_{j-1})^{2H}+\varepsilon\big]\bigg) \prod^{m}_{j=1} |x_j|^{k(p_j+\overline{p}_{j-1})}\, dx\\
&\leq C_{T, k, m,H} \sum_{\mathcal{S}} \prod^{m}_{j=1}  \Big[(s_j-s_{j-1})^{2H}+\varepsilon\Big]^{-\frac{1+k(p_j+\overline{p}_{j-1})}{2}}.
\end{align*}
This gives the desired result.
\end{proof}

\begin{lemma} \label{lma2} For any $p\geq 1$ and $m\in\N$, there exists a positive constant $C_{T,m, H,p}$ depending on $T$, $m$, $H$ and $p$ such that
\begin{align*} 
&\bigg(\int_{[0,T]^m_<} \exp\Big(-\frac{1}{2}\Var\big(\sum\limits^m_{j=1}y_j \cdot X^{H}_{u_i}\big)\Big)\, du\bigg)^p\\
&\qquad\qquad \leq C^m_{T,m,H,p}  \int_{[0,T]^m_<} \exp\bigg(-\frac{\kappa_{T, m,H}}{2}\sum\limits^m_{j=1}\Big|\sum\limits^m_{i=j}y_i\Big|^2 (\Delta u_j)^{\frac{2H}{p}}\bigg)\, du,
\end{align*}
where 
\[
[0,T]^m_<=\big\{0=u_0<u_1<u_2\dots<u_m<T\big\}
\]
and $\Delta u_j=u_j-u_{j-1}$ for $j=1,2,\dots,m$.
\end{lemma}
\begin{proof}
For any $a\geq 0$ and $T>0$, let $f_T(a)=\frac{\left(\int^T_0 e^{-a\, v^{2H}}\, dv\right)^p}{ \int^{T/m}_0 e^{-a\, v^{\frac{2H}{p}}}\, dv}$. Clearly, $f_T(\cdot)$ is continuous on $[0,\infty)$, $\lim\limits_{a\to\infty}f_T(a)=\frac{\Gamma^p(\frac{1}{2H})}{p(2H)^{p-1}\Gamma(\frac{p}{2H})}$ and $f_T(0)=mT^{p-1}$. Therefore, there exists a positive constant $C_{T, m, H,p}$ depending only on $T$, $m$, $H$ and $p$ such that 
\[
\left(\int^T_0 e^{-a\, v^{2H}}\, dv\right)^p\leq C_{T, m, H,p}\int^{T/m}_0 e^{-a\, v^{\frac{2H}{p}}}\, dv,\quad \text{for all}\, a\geq 0.
\]

Using local nondeterminism property (\ref{ln}) and the above inequality, 
\begin{align*}
&\bigg(\int_{[0,T]^m_<} \exp\Big(-\frac{1}{2}\Var\big(\sum\limits^m_{j=1}y_j \cdot X^{H}_{u_i}\big)\Big)\, du\bigg)^p\\
&\qquad \leq \bigg(\int_{[0,T]^m_<} \exp\Big(-\frac{\kappa_{T,m,H}}{2}\sum\limits^m_{j=1}\Big|\sum\limits^m_{i=j}y_i\Big|^2 (\Delta u_j)^{2H}\Big)\, du\bigg)^p\\
&\qquad \leq \prod\limits^m_{j=1} \bigg(\int^T_0 \exp\bigg(-\frac{\kappa_{T,m,H}}{2}\Big|\sum\limits^m_{i=j}y_i\Big|^2 v^{2H}\bigg) \, dv\bigg)^p\\
&\qquad \leq C^m_{T,m,H,p} \prod\limits^m_{j=1} \int^{T/m}_0 \exp\bigg(-\frac{\kappa_{T,m,H}}{2}\Big|\sum\limits^m_{i=j}y_i\Big|^2 v^{\frac{2H}{p}}\bigg)\, dv\\
&\qquad \leq C^m_{T,m,H,p}  \int_{[0,T]^m_<} \exp\bigg(-\frac{\kappa_{T,m,H}}{2}\sum\limits^m_{j=1}\Big|\sum\limits^m_{i=j}y_i\Big|^2 (\Delta u_j)^{\frac{2H}{p}}\bigg)\, du.
\end{align*}
This completes the proof.
\end{proof}

\bigskip

\section{Proofs of main results.}

In this section, we will give proofs of Theorems \ref{thm1} and \ref{thm2}.

\subsection{Proof of Theorem \ref{thm1}}

\begin{proof}
Recall the definition of $L^{(\mathbf{k})}_{\varepsilon}(T,x)$ in (\ref{epsilon}). Using Fourier transform, 
\begin{align*}
L^{(\mathbf{k})}_{\varepsilon}(T,x)
&=\frac{\iota^{|\mathbf{k}|}}{(2\pi)^d} \int^T_0\int^T_0\int_{\R^d}  \prod^d_{i=1}y^{k_i}_i\, e^{\iota y\cdot (X^{H_1}_u-\widetilde{X}^{H_2}_v-x)}e^{-\frac{\varepsilon|y|^2}{2}} \, dy\, du\, dv.
\end{align*}
Fix an even integer $m\geq 1$ and denote $T_m=[0,T]^{2m}$. We have 
\begin{align*} \label{e2}
\E\Big[ |L^{(\mathbf{k})}_{\varepsilon}(T,x)|^m\Big]   \nonumber
&\leq \frac{1}{(2\pi)^{md}}\int_{T_m}\int_{\R^{md}} \exp\bigg\{-\frac{1}{2} \E\Big[\sum^m_{j=1} y_j \cdot(X^{H_1}_{s_j}-\widetilde{X}^{H_2}_{t_j})\Big]^2 \bigg\}\\  
&\qquad\qquad \times \exp\Big\{-\frac{\varepsilon}{2}\sum^m_{j=1}|y_j|^2\Big\} \prod^m_{j=1} \Big(\prod^d_{i=1}| y_{j,i}|^{k_i}\Big)\, dy\, dt\, ds.
\end{align*}
Let $\mathscr{P}_m$ be the set  of all permutations of $\{1,2,\dots, m\}$ and 
\[
D^{m}_T=\big\{ u\in [0, T]^m: 0<u_1<\dots <u_m<T\big\}.
\] 
Then
\begin{align*}
&\E\Big[ |L^{(\mathbf{k})}_{\varepsilon}(T,x)|^m\Big] \\
&\leq \frac{m!}{(2\pi)^{md}}\sum_{\sigma\in \mathscr{P}_m}\int_{D^{m}_T\times D^{m}_T}\int_{\R^{md}} \exp\bigg\{-\frac{1}{2} \E\Big[\sum^m_{j=1} y_j \cdot (X^{H_1}_{s_j}-\widetilde{X}^{H_2}_{t_{\sigma(j)}})\Big]^2 \bigg\} \prod^m_{j=1} |y_{j}|^{|\mathbf{k}|}\, dy\, dt\, ds.
\end{align*}

Using H\"{o}lder inequality,
\begin{align*} 
&\int_{\R^{md}} \exp\bigg\{-\frac{1}{2} \E\Big[\sum^m_{j=1} y_j \cdot(X^{H_1}_{s_j}-\widetilde{X}^{H_2}_{t_{\sigma(j)}})\Big]^2 \bigg\} \prod^m_{j=1}|y_{j}|^{|\mathbf{k}|}\, dy \nonumber \\
&\qquad \leq \bigg(\int_{\R^{md}} \exp\Big\{-\frac{H_1+H_2}{2H_2} \Var\big(\sum^m_{j=1} y_j \cdot X^{H_1}_{s_j}\big) \Big\} \prod^m_{j=1} |y_{j}|^{|\mathbf{k}|}\, dy\bigg)^{\frac{H_2}{H_1+H_2}} \nonumber \\
&\qquad\qquad \times \bigg(\int_{\R^{md}} \exp\Big\{-\frac{H_1+H_2}{2H_1} \Var\big(\sum^m_{j=1} y_j \cdot X^{H_2}_{t_{\sigma(j)}}\big) \Big\} \prod^m_{j=1}|y_{j}|^{|\mathbf{k}|}\, dy\bigg)^{\frac{H_1}{H_1+H_2}} \nonumber \\
&\qquad\leq \bigg(\int_{\R^{md}} \exp\Big\{-\frac{1}{2} \Var\big(\sum^m_{j=1} y_j \cdot X^{H_1}_{s_j}\big) \Big\} \prod^m_{j=1} |y_{j}|^{|\mathbf{k}|}\, dy\bigg)^{\frac{H_2}{H_1+H_2}} \nonumber \\
&\qquad\qquad \times \bigg(\int_{\R^{md}} \exp\Big\{-\frac{1}{2} \Var\big(\sum^m_{j=1} y_j \cdot X^{H_2}_{t_{\sigma(j)}}\big) \Big\} \prod^m_{j=1} |y_{j}|^{|\mathbf{k}|}\, dy\bigg)^{\frac{H_1}{H_1+H_2}}.
\end{align*}

Now by Lemma \ref{lma1} and the inequality $\big(\sum\limits^n_{i=1} a_i\big)^{\alpha}\leq \sum\limits^n_{i=1} a_i^{\alpha}$
for any $n\in\N$, $a_i\geq 0$ and $\alpha\in[0,1]$, we have
\begin{align}  \label{holder}
\E\Big[ |L^{(\mathbf{k})}_{\varepsilon}(T,x)|^m\Big]
&\leq c_1 \bigg( \sum_{\mathcal{S}} \int_{D^{m}_T}  \prod^{m}_{j=1}(s_j-s_{j-1})^{-\frac{H_1H_2}{H_1+H_2}d-\frac{H_1H_2}{H_1+H_2}|\mathbf{k}|(p_{j}+\overline{p}_{j-1})}ds\bigg)^2,
\end{align}
where $\mathcal{S}=\big\{p_{j}, \overline{p}_{j}:  p_{j}\in \{0,1\},\,  p_{j}+\overline{p}_{j}=1,\, j=1,\dots,m-1, p_{m}=1\big\}$.

Note that $p_{j}+\overline{p}_{j-1}\in\{0,1,2\}$ for $j=1,\dots,m$. Hence, when $\frac{H_1H_2}{H_1+H_2}(d+2|\mathbf{k}|)<1$,
\begin{align*} \label{m}
\E\Big[ |L^{(\mathbf{k})}_{\varepsilon}(T,x)|^m\Big]<\infty
\end{align*}
for all $\varepsilon>0$.

Observe that
\begin{align*}
&\E\Big[ |L^{(\mathbf{k})}_{\varepsilon}(T,x)-L^{(\mathbf{k})}_{\eta}(T,x)|^m\Big]\\
&\leq \frac{1}{(2\pi)^{md}}\int_{T_m}\int_{\R^{md}} \exp\bigg\{-\frac{1}{2} \E\Big[\sum^m_{j=1} y_j \cdot(X^{H_1}_{s_j}-\widetilde{X}^{H_2}_{t_j})\Big]^2 \bigg\}\\  
&\qquad\qquad \times \prod^m_{j=1}\bigg|\exp\Big\{-\frac{\varepsilon}{2}|y_j|^2\Big\}-\exp\Big\{-\frac{\eta}{2}|y_j|^2\Big\}\bigg| \prod^m_{j=1} \Big(\prod^d_{i=1}| y_{j,i}|^{k_i}\Big)\, dy\, dt\, ds
\end{align*} 
and 
\begin{align*}
\int_{T_m}\int_{\R^{md}} \exp\bigg\{-\frac{1}{2} \E\Big[\sum^m_{j=1} y_j \cdot(X^{H_1}_{s_j}-\widetilde{X}^{H_2}_{t_j})\Big]^2 \bigg\}\bigg| \prod^m_{j=1} \Big(\prod^d_{i=1}| y_{j,i}|^{k_i}\Big)\, dy\, dt\, ds<\infty
\end{align*} 
when $\frac{H_1H_2}{H_1+H_2}(d+2|\mathbf{k}|)<1$.  Now, by the dominated convergence theorem, we can easily obtain that  the $\mathbf{k}$-th derivative of the local time $L^{(\mathbf{k})}(T,x)$ exists in $L^p$ for any $p\in[1,\infty)$ if $\frac{H_1H_2}{H_1+H_2}(d+2|\mathbf{k}|)<1$.

In the sequel, we show the H\"{o}lder continuity of $L^{(\mathbf{k})}(T,x)$ in time and space variables.  For H\"{o}lder continuity in the space variable, using Fourier transform,
\begin{align*}
&L^{(\mathbf{k})}_{\varepsilon}(T,z+x)-L^{(\mathbf{k})}_{\varepsilon}(T,x)\\
&=\frac{\iota^{|\mathbf{k}|}}{(2\pi)^d} \int^T_0\int^T_0\int_{\R^d}  \prod^d_{i=1}y^{k_i}_i\, e^{\iota y\cdot (X^{H_1}_u-\widetilde{X}^{H_2}_v)}(e^{-\iota y\cdot (z+x)}-e^{-\iota y\cdot x})e^{-\frac{\varepsilon|y|^2}{2}} \, dy\, du\, dv.
\end{align*}
Then, for any even integer $m$,
\begin{align*} 
&\E\Big[ |L^{(\mathbf{k})}_{\varepsilon}(T,z+x)-L^{(\mathbf{k})}_{\varepsilon}(T,x)|^m\Big]\\  
&\leq \int_{T_m}\int_{\R^{md}} \exp\Big\{-\frac{1}{2} \E\big[\sum^m_{j=1} y_j \cdot(X^{H_1}_{s_j}-\widetilde{X}^{H_2}_{t_j})\big]^2 \Big\}\\  
&\qquad\qquad \times \exp\Big\{-\frac{\varepsilon}{2}\sum^m_{j=1}|y_j|^2\Big\} \prod^m_{j=1} \big(|e^{-\iota y_j\cdot z}-1| |y_{j}|^{|\mathbf{k}|}\big)\, dy\, dt\, ds.
\end{align*}
Note that $|e^{-\iota y_j\cdot z}-1|\leq c_{\alpha} |z|^{\alpha} |y_{j}|^{\alpha}$ for any $\alpha\in[0,1]$. Hence
\[
\prod^m_{j=1} \Big(|e^{-\iota y_j\cdot z_j}-1| |y_{j}|^{|\mathbf{k}|}\Big)\leq  c^m_{\alpha}  |z|^{m\alpha} \prod^m_{j=1}  |y_{j}|^{|\mathbf{k}|+\alpha}.
\]
Similarly, for any $\alpha\in[0,1\wedge (\frac{1}{H_1}+\frac{1}{H_2}-2|\mathbf{k}|-d))$, 
\begin{align*} 
\E\Big[ |L^{(\mathbf{k})}_{\varepsilon}(T,z+x)-L^{(\mathbf{k})}_{\varepsilon}(T,x)|^m\Big]\leq c_2 |z|^{m\alpha}.
\end{align*}
Therefore,
\begin{align*} 
&\E\Big[ |L^{(\mathbf{k})}(T,z+x)-L^{(\mathbf{k})}(T,x)|^m\Big]\\
&\leq 3^m\lim_{\varepsilon\downarrow 0}\E\Big[ |L^{(\mathbf{k})}(T,z+x)-L^{(\mathbf{k})}_{\varepsilon}(T,z+x)|^m\Big]+3^m\limsup_{\varepsilon\downarrow 0}\E\Big[ |L^{(\mathbf{k})}_{\varepsilon}(T,z+x)-L^{(\mathbf{k})}_{\varepsilon}(T,x)|^m\Big]\\
&\qquad+3^m\lim_{\varepsilon\downarrow 0}\E\Big[ |L^{(\mathbf{k})}_{\varepsilon}(T,x)-L^{(\mathbf{k})}(T,x)|^m\Big]\\
&\leq c_3 |z|^{m\alpha}.
\end{align*}
The desired $\theta_1$-H\"{o}lder continuity in the space variable follows from the Kolmogorov's continuity criterion.

For H\"{o}lder continuity in the time variable, we see that
\begin{align*}
&L^{(\mathbf{k})}_{\varepsilon}(T+h,x)-L^{(\mathbf{k})}_{\varepsilon}(T,x)\\
&=\int^{T+h}_0\int^T_0 p^{(\mathbf{k})}_{\varepsilon}(X^{H_1}_t-\widetilde{X}^{H_2}_s-x)\, ds\, dt+\int^{T}_0\int^{T+h}_{T} p^{(\mathbf{k})}_{\varepsilon}(X^{H_1}_t-\widetilde{X}^{H_2}_s-x)\, ds\, dt\\
&\qquad\qquad+\int^{T+h}_{T}\int^{T+h}_T p^{(\mathbf{k})}_{\varepsilon}(X^{H_1}_t-\widetilde{X}^{H_2}_s-x)\, ds\, dt\\
&=:I_{\varepsilon,1}+I_{\varepsilon,2}+I_{\varepsilon,3}.
\end{align*}
It suffices to show that, for any even integer $m$, $\E|I_{\varepsilon,1}|^m$ is less than a constant multiple of $h^{m\big(1-\frac{H_1H_2}{H_1+H_2}(|\mathbf{k}|+d)\big)}$. Let $T^h_m=[T,T+h]^m\times[0,T]^m$. By Fourier transform,  
\begin{align*}  
\E|I_{\varepsilon,1}|^m
&\leq \frac{1}{(2\pi)^{md}}\int_{T^h_m}\int_{\R^{md}} \exp\bigg\{-\frac{1}{2} \E\Big[\sum^m_{j=1} y_j \cdot(X^{H_1}_{s_j}-\widetilde{X}^{H_2}_{t_j})\Big]^2 \bigg\}\\  
&\qquad\qquad \times \exp\Big\{-\frac{\varepsilon}{2}\sum^m_{j=1}|y_j|^2\Big\} \prod^m_{j=1} |y_{j}|^{|\mathbf{k}|}\, dy\, dt\, ds.
\end{align*}
Let $D^{m}_{T,h}=\big\{ u\in [T, T+h]^m: T<u_1<\dots <u_m<T+h\big\}$. Now using H\"{o}lder inequality as in obtaining (\ref{holder}) and then Lemmas \ref{a1} and \ref{a2},
\begin{align*}  
\E|I_{\varepsilon,1}|^m
&\leq c_3  \sum_{\mathcal{S}} \int_{D^{m}_{T,h}}  \prod^{m}_{j=1}(s_j-s_{j-1})^{-\frac{H_1H_2}{H_1+H_2}d-\frac{H_1H_2}{H_1+H_2}|\mathbf{k}|(p_{j}+\overline{p}_{j-1})}ds\\
&\leq c_4 \sum_{\mathcal{S}} h^{\sum\limits^m_{j=1}\Big[1-\frac{H_1H_2}{H_1+H_2}d-\frac{H_1H_2}{H_1+H_2}|\mathbf{k}|(p_{j}+\overline{p}_{j-1})\Big]}\\
&\leq c_5\, h^{m\big(1-\frac{H_1H_2}{H_1+H_2}(|\mathbf{k}|+d)\big)}.
\end{align*}
Therefore, 
\begin{align*}
&\E\Big[\big|L^{(\mathbf{k})}(T+h,x)-L^{(\mathbf{k})}(T,x)\big|^m\Big]\\
&\leq 3^m \lim_{\varepsilon\downarrow 0} \E\Big[\big|L^{(\mathbf{k})}(T+h,x)-L^{(\mathbf{k})}_{\varepsilon}(T+h,x)\big|^m\Big]+3^m\limsup_{\varepsilon\downarrow 0} \E\Big[\big|L^{(\mathbf{k})}_{\varepsilon}(T+h,x)-L^{(\mathbf{k})}_{\varepsilon}(T,x)\big|^m\Big]\\
&\qquad+3^m\lim_{\varepsilon\downarrow 0} \E\Big[\big|L^{(\mathbf{k})}_{\varepsilon}(T,x)-L^{(\mathbf{k})}(T,x)\big|^m\Big]\\
&\leq c_6\, h^{m\big(1-\frac{H_1H_2}{H_1+H_2}(|\mathbf{k}|+d)\big)}.
\end{align*}
By Kolmogorov's continuity criterion, we get the $\theta_2$-H\"{o}lder continuity in the time variable. 

This completes the proof.
\end{proof}

\bigskip

\subsection{Proof of Theorem \ref{thm2}}

\begin{proof} We divide the proof into several steps.  

\noindent
{\bf Step 1.} Recall that 
\begin{align} \label{2m}
\E\Big[ |L^{(\mathbf{k})}_{\varepsilon}(T,0)|^2\Big]   \nonumber
&=\frac{(-1)^{|\mathbf{k}|}}{(2\pi)^{2d}}\int_{[0,T]^4}\int_{\R^{2d}}  I_2(H_1,s,x)  \widetilde{I}_2(H_2,t,x)\\
&\qquad \times\exp\Big\{-\frac{\varepsilon}{2}(|x_1|^2+|x_2|^2)\Big\}  \prod^d_{i=1}x^{k_i}_{2,i} \prod^d_{i=1}x^{k_i}_{1,i} \, dx\, dt\, ds,
\end{align}
where $I_2(H,t,x)=\exp\Big\{-\frac{1}{2} \E\big[ x_2\cdot X^{H}_{t_2}+x_1\cdot X^{H}_{t_1}\big]^2\Big\}$ and 
\begin{align*}
\widetilde{I}_2(H,t,x)=\exp\Big\{-\frac{1}{2} \E\big[ x_2\cdot \widetilde{X}^{H}_{t_2}+x_1\cdot \widetilde{X}^{H}_{t_1}\big]^2\Big\}.
\end{align*}
We define $F^{(\mathbf{k})}_{T,\varepsilon}$ by replacing $\prod\limits^d_{i=1}x^{k_i}_{1,i}$ in (\ref{2m}) with $\prod\limits^d_{i=1}(-x_{2,i})^{k_i}$.  That is,
\begin{align*} \label{2m0}
F^{(\mathbf{k})}_{T,\varepsilon}  \nonumber
&=\frac{1}{(2\pi)^{2d}}\int_{[0,T]^4}\int_{\R^{2d}}  I_2(H_1,s,x)  \widetilde{I}_2(H_2,t,x)\\
&\qquad\qquad \times\exp\Big\{-\frac{\varepsilon}{2}(|x_1|^2+|x_2|^2)\Big\}  \prod^d_{i=1}x^{2k_i}_{2,i}\, dx\, dt\, ds.
\end{align*}

Note that  
\[
\min\Big\{I_2(H,t,y),\widetilde{I}_2(H,t,y)\Big\}\geq \exp\Big\{-c_1 \big(|x_2|^2t_2^{2H}+|x_1|^2t^{2H}_1\big)\Big\}.
\]
Hence,
\begin{align*}
F^{(\mathbf{k})}_{T,\varepsilon}
&\geq c_2\, \int_{[0,T]^4}\int_{\R^{2d}}  \exp\Big\{-|x_2|^2(t^{2H_2}_2+s^{2H_1}_2+\varepsilon)-|x_1|^2(t^{2H_2}_1+s^{2H_1}_1+\varepsilon)\Big\}\prod^d_{i=1}x^{2k_i}_{2,i} dx\, dt\, ds\\
&=c_3\, \int_{[0,T]^4}   (t^{2H_2}_2+s^{2H_1}_2+\varepsilon)^{-\frac{d}{2}} (t^{2H_2}_1+s^{2H_1}_1+\varepsilon)^{-\frac{d}{2}-|\mathbf{k}|}\, dt\, ds.
\end{align*}

By Lemma \ref{a3}, 
\begin{align} \label{fkt}
F^{(\mathbf{k})}_{T,\varepsilon}\geq 
\left\{\begin{array}{ll} 
c_4\,\varepsilon^{\frac{H_1+H_2}{H_1H_2}-d}&   \text{if}\; \frac{H_1H_2}{H_1+H_2}d>1,  |\mathbf{k}|=0\\  
c_4\,\ln^2(1+\varepsilon^{-\frac{1}{2}}) & \text{if}\; \frac{H_1H_2}{H_1+H_2}d=1, |\mathbf{k}|=0 \\
c_4\,\ln(1+\varepsilon^{-\frac{1}{2}})\varepsilon^{\frac{H_1+H_2}{2H_1H_2}-\frac{d}{2}-|\mathbf{k}|} &   \text{if}\; \frac{H_1H_2}{H_1+H_2}d=1, \frac{H_1H_2}{H_1+H_2}(2|\mathbf{k}|+d)>1\\
c_4\,\varepsilon^{\frac{H_1+H_2}{2H_1H_2}-\frac{d}{2}-|\mathbf{k}|} &   \text{if}\; \frac{H_1H_2}{H_1+H_2}d<1,  \frac{H_1H_2}{H_1+H_2}(2|\mathbf{k}|+d)>1\\  
c_4\,\ln(1+\varepsilon^{-\frac{1}{2}}) &   \text{if}\;  \frac{H_1H_2}{H_1+H_2}d<1, \frac{H_1H_2}{H_1+H_2}(2|\mathbf{k}|+d)=1. 
\end{array} \right.
\end{align}

\noindent
{\bf Step 2.} We estimate $\Big|\E\Big[ |L^{(\mathbf{k})}_{\varepsilon}(T,0)|^2\Big]-F^{(\mathbf{k})}_{T,\varepsilon}\Big|$. It is easy to see that $\Big|\E\Big[ |L^{(\mathbf{0})}_{\varepsilon}(T,0)|^2\Big]-F^{(\mathbf{0})}_{T,\varepsilon}\Big|=0$. So it suffices to consider the case $|\mathbf{k}|\geq 1$ in the sequal. Recall that $\mathscr{P}_2$ is the set of all permutations of $\{1,2\}$. That is, 
\[
\mathscr{P}_2=\Big\{\sigma_1,\sigma_2: \sigma_1(1)=\sigma_2(2)=1, \sigma_1(2)=\sigma_2(1)=2\Big\}.
\]

For $\sigma\in\mathscr{P}_2$, define
\begin{align*}
\widetilde{J}^{\sigma}_2(H,t,y)=\exp\Big\{-\frac{1}{2} \text{Var}\big[ (y_{\sigma(2)}-y_{\sigma(2)+1})\cdot (\widetilde{X}^{H}_{t_2}-\widetilde{X}^{H}_{t_1})+y_{1}\cdot \widetilde{X}^{H}_{t_1}\big]\Big\}
\end{align*}
with the convention $y_3=0$ and 
\begin{align*}
J_2(H,t,y)=\exp\Big\{-\frac{1}{2} \text{Var}\big[ y_2\cdot (X^{H}_{t_2}-X^{H}_{t_1})+y_1\cdot X^{H}_{t_1}\big]\Big\}.
\end{align*}

Making the change of variables $y_2=x_2$ and $y_1=x_2+x_1$ gives 
\begin{align*}
&\E\Big[ |L^{(\mathbf{k})}_{\varepsilon}(T,0)|^2\Big]-F^{(\mathbf{k})}_{T,\varepsilon}\\
&=2\frac{(-1)^{|\mathbf{k}|}}{(2\pi)^{2d}}\sum_{\sigma\in\mathscr{P}_2}\int_{D^2_T\times D^2_T}\int_{\R^{2d}}  J_2(H_1,s,y)  \widetilde{J}^{\sigma}_2(H_2,t,y)\exp\Big\{-\frac{\varepsilon}{2}\big(|y_1-y_2|^2+|y_2|^2\big)\Big\} \\
&\qquad \times \prod^d_{i=1}y^{k_i}_{2,i} \Big(\prod^d_{i=1}(y_{1,i}-y_{2,i})^{k_i}- \prod^d_{i=1}(-y_{2,i})^{k_i}\Big) \, dy\, dt\, ds,
\end{align*}
where $D^2_T=\big\{0<t_1<t_2<T\big\}$.

Set 
\[
\Delta^{\mathbf{k}, d}(y)=\prod^d_{i=1}y^{k_i}_{2,i}\Big(\prod^d_{i=1}(y_{1,i}-y_{2,i})^{k_i}-\prod^d_{i=1}(-y_{2,i})^{k_i}\Big).
\]
It is easy to see that 
\[
|\Delta^{\mathbf{k}, d}(y)|\leq c_5\sum^{|\mathbf{k}|}_{\ell=1} |y_2|^{2|\mathbf{k}|-\ell}|y_1|^{\ell}.
\]

For $\ell=1,\dots, |\mathbf{k}|$, define 
\[
F^{\ell}(y)=\exp\Big\{-\frac{\varepsilon}{2}(|y_1-y_2|^2+|y_2|^2)\Big\} |y_2|^{2|\mathbf{k}|-\ell}|y_1|^{\ell}.
\] 
Then,
\begin{align} \label{dva}
\Big|\E\Big[ |L^{(\mathbf{k})}_{\varepsilon}(T,0)|^2\Big]-F^{(\mathbf{k})}_{T,\varepsilon}\Big|\leq  c_5\sum^{|\mathbf{k}|}_{\ell=1} (K^{\ell}_{\varepsilon,1}+K^{\ell}_{\varepsilon, 2}),
\end{align}
where 
\begin{align*}
K^{\ell}_{\varepsilon,i}
&=\int_{D^2_T\times D^2_T}\int_{\R^{2d}} J_2(H_1,s,y)  \widetilde{J}^{\sigma_i}_2(H_2,t,y)F^{\ell}(y)\, dy\, dt\, ds, \quad i=1,2.
\end{align*}

\noindent
{\bf Step 3.}  We estimate $K^{\ell}_{\varepsilon,1}$ and $K^{\ell}_{\varepsilon,2}$ for $\ell=1,\dots, |\mathbf{k}|$. Using the local nondeterminism property (\ref{ln}) and Lemma \ref{a4},
\begin{align*}
K^{\ell}_{\varepsilon,1}
&\leq c_6\int_{D^2_T\times D^2_T}\int_{\R^{2d}} \exp\Big\{-|y_2|^2\{(t_2-t_1)^{2H_2}+(s_2-s_1)^{2H_1}+\varepsilon\} \Big\}\\
&\qquad\qquad \times \exp\Big\{-|y_1|^2\{t_1^{2H_2}+s_1^{2H_1}+\varepsilon\} \Big\} |y_2|^{2|\mathbf{k}|-\ell}|y_1|^{\ell}\, dy\, dt\, ds\\
&\leq c_7 \int_{D^2_T\times D^2_T}\{(t_2-t_1)^{2H_2}+(s_2-s_1)^{2H_1}+\varepsilon\}^{-\frac{2|\mathbf{k}|-\ell+d}{2}}   \{t_1^{2H_2}+s_1^{2H_1}+\varepsilon\}^{-\frac{\ell+d}{2}}\, dt\, ds\\
&\leq c_7 \int_{[0,T]^2}(u^{2H_2}+v^{2H_1}+\varepsilon)^{-\frac{2|\mathbf{k}|-\ell+d}{2}}\, du\,dv  \int_{[0,T]^2} (u^{2H_2}+v^{2H_1}+\varepsilon)^{-\frac{\ell+d}{2}}\, du\, dv.
\end{align*}

For $\ell=1,\dots, |\mathbf{k}|$, by Lemma \ref{a3},  we could obtain that 
\begin{align}   \label{kl1}
K^{\ell}_{\varepsilon,1}\leq c_8\, h^{d,|\mathbf{k}|,1}_{H_1,H_2}(\varepsilon),
\end{align}
where 
\begin{align}   \label{h1}
h^{d,|\mathbf{k}|,1}_{H_1,H_2}(\varepsilon)= 
\left\{\begin{array}{ll} 
\varepsilon^{\frac{H_1+H_2}{2H_1H_2}-\frac{d}{2}-|\mathbf{k}|} &   \text{if}\; \frac{H_1H_2}{H_1+H_2}d=1, \frac{H_1H_2}{H_1+H_2}(2|\mathbf{k}|+d)>1\\
\varepsilon^{\frac{H_1+H_2}{2H_1H_2}-\frac{d}{2}-|\mathbf{k}|+\beta} &   \text{if}\; \frac{H_1H_2}{H_1+H_2}d<1,  \frac{H_1H_2}{H_1+H_2}(2|\mathbf{k}|+d)>1\\  
1  &   \text{if}\;  \frac{H_1H_2}{H_1+H_2}d<1, \frac{H_1H_2}{H_1+H_2}(2|\mathbf{k}|+d)=1,
\end{array} \right.
\end{align}
where $\beta=\frac{1}{4}\big\{1\wedge (\frac{H_1+H_2}{H_1H_2}-d)\wedge (d+2|\mathbf{k}|-\frac{H_1+H_2}{H_1H_2})\big\}$.

We next estimate $K^{\ell}_{\varepsilon,2}$. Using H\"{o}lder inequality,
\begin{align*} 
K^{\ell}_{\varepsilon,2} \nonumber
&\leq \bigg(\int_{\R^{2d}} F^{\ell}(y)  \Big(\int_{D^2_T} J_2(H_1,s,y) ds\Big)^{\frac{H_1+H_2}{H_2}} dy\bigg)^{\frac{H_2}{H_1+H_2}} \\
&\qquad\qquad  \times \bigg(\int_{\R^{2d}} F^{\ell}(y)\Big(\int_{D^2_T} \widetilde{J}^{\sigma_2}_2(H_2,t,y)\, dt\Big)^{\frac{H_1+H_2}{H_1}} dy\bigg)^{\frac{H_1}{H_1+H_2}}.
\end{align*}

Then, by Lemmas \ref{lma2} and \ref{a4},
\begin{align*}
&\int_{\R^{2d}} F^{\ell}(y)  \Big(\int_{D^2_T} J_2(H_1,s,y) ds\Big)^{\frac{H_1+H_2}{H_2}} dy\\
&\leq \int_{\R^{2d}} \int_{D^2_T} \exp\Big\{-c_9\big(|y_2|^2\{(s_2-s_1)^{\frac{2H_1H_2}{H_1+H_2}}+\varepsilon\}+|y_1|^2 \{s_1^{\frac{2H_1H_2}{H_1+H_2}}+\varepsilon\}\big)\Big\}|y_2|^{2|\mathbf{k}|-\ell}|y_1|^{\ell}\, ds\, dy\\
&\leq c_{10}\int^T_0 (u^{\frac{2H_1H_2}{H_1+H_2}}+\varepsilon)^{-\frac{2|\mathbf{k}|-\ell+d}{2}} du  \int^T_0 (u^{\frac{2H_1H_2}{H_1+H_2}}+\varepsilon)^{-\frac{\ell+d}{2}} du\\
&\leq c_{11}\, h^{d,|\mathbf{k}|,1}_{H_1,H_2}(\varepsilon).
\end{align*}

Making the change of variables gives $y_1-y_2=x_2$ and $y_1=x_1$ gives
\begin{align*}
&\int_{\R^{2d}} F^{\ell}(y)\Big(\int_{D^2_T} \widetilde{J}^{\sigma_2}_2(H_2,t,y)\, dt\Big)^{\frac{H_1+H_2}{H_1}} dy\\
&=\int_{\R^{2d}} \exp\Big\{-\frac{\varepsilon}{2}(|x_1-x_2|^2+|x_2|^2)\Big\} |x_1-x_2|^{2|\mathbf{k}|-\ell}|x_1|^{\ell}  \Big(\int_{D^2_T} \widetilde{J}^{\sigma_1}_2(H_2,t,x)\, dt\Big)^{\frac{H_1+H_2}{H_1}} dx\\
&\leq \, \int_{\R^{2d}}\int_{D^2_T}  \exp\Big\{-c_{12}\big(|x_2|^2\{(t_2-t_1)^{\frac{2H_1H_2}{H_1+H_2}}+\varepsilon\}+|x_1|^2\{ t_1^{\frac{2H_1H_2}{H_1+H_2}}+\varepsilon\}\big)\Big\}\\
&\qquad\qquad\qquad \times \sum^{2|\mathbf{k}|-\ell}_{j=0}|x_1|^{j+\ell} |x_2|^{2|\mathbf{k}|-\ell-j}   \,dt\, dx\\
&\leq  c_{13} \sum^{2|\mathbf{k}|-\ell}_{j=0}  \int^T_0 (u^{\frac{2H_1H_2}{H_1+H_2}}+\varepsilon)^{-\frac{2|\mathbf{k}|-\ell-j+d}{2}} du  \int^T_0 (u^{\frac{2H_1H_2}{H_1+H_2}}+\varepsilon)^{-\frac{\ell+j+d}{2}} du\\
&\leq  c_{14}\,  h^{d,|\mathbf{k}|,2}_{H_1,H_2}(\varepsilon),
\end{align*}
where in the last inequality we used Lemma \ref{a3} and  
\begin{align} \label{h2}
h^{d,|\mathbf{k}|,2}_{H_1,H_2}(\varepsilon)= 
\left\{\begin{array}{ll} 
\ln(1+\varepsilon^{-\frac{1}{2}}) \varepsilon^{\frac{H_1+H_2}{2H_1H_2}-\frac{d}{2}-|\mathbf{k}|} &   \text{if}\; \frac{H_1H_2}{H_1+H_2}d=1, \frac{H_1H_2}{H_1+H_2}(2|\mathbf{k}|+d)>1\\
\varepsilon^{\frac{H_1+H_2}{2H_1H_2}-\frac{d}{2}-|\mathbf{k}|} &   \text{if}\; \frac{H_1H_2}{H_1+H_2}d<1,  \frac{H_1H_2}{H_1+H_2}(2|\mathbf{k}|+d)>1\\  
\ln(1+\varepsilon^{-\frac{1}{2}})  &   \text{if}\;  \frac{H_1H_2}{H_1+H_2}d<1, \frac{H_1H_2}{H_1+H_2}(2|\mathbf{k}|+d)=1. 
\end{array} \right.
\end{align}

Therefore, 
\begin{align} \label{kl2}
K^{\ell}_{\varepsilon,2} 
&\leq c_{15}\,    \big( h^{d,|\mathbf{k}|,1}_{H_1,H_2}(\varepsilon)\big)^{\frac{H_1}{H_1+H_2}} \big(h^{d,|\mathbf{k}|,2}_{H_1,H_2}(\varepsilon)\big)^{\frac{H_2}{H_1+H_2}}.
\end{align}

\noindent
{\bf Step 4.}  We show the divergence of $\E\Big[ |L^{(\mathbf{k})}_{\varepsilon}(T,0)|^2\Big]$ as $\varepsilon$ tends to $0$.  Combining inequalities (\ref{dva}), (\ref{kl1}) and (\ref{kl2}) gives 
\begin{align*}
\Big| \E\Big[ |L^{(\mathbf{k})}_{\varepsilon}(T,0)|^2\Big]-F^{(\mathbf{k})}_{T,\varepsilon}\Big|
&\leq  c_{16}\, 1_{\N}(|\mathbf{k}|) \big( h^{d,|\mathbf{k}|,1}_{H_1,H_2}(\varepsilon)\big)^{\frac{H_1}{H_1+H_2}} \big(h^{d,|\mathbf{k}|,2}_{H_1,H_2}(\varepsilon)\big)^{\frac{H_2}{H_1+H_2}}.
\end{align*}
Recall the inequality (\ref{fkt}), definition of $h^{d,|\mathbf{k}|,1}_{H_1,H_2}(\varepsilon)$ in (\ref{h1}) and definition of $h^{d,|\mathbf{k}|,2}_{H_1,H_2}(\varepsilon)$ in (\ref{h2}). We finally have 
\begin{align*}
\E\big[ |L^{(\mathbf{k})}_{\varepsilon}(T,0)|^2\big] \geq 
\left\{\begin{array}{ll} 
c_{17}\,\varepsilon^{\frac{H_1+H_2}{H_1H_2}-d}&   \text{if}\; \frac{H_1H_2}{H_1+H_2}d>1, |\mathbf{k}|=0\\  
c_{17}\,\ln^2(1+\varepsilon^{-\frac{1}{2}}) & \text{if}\; \frac{H_1H_2}{H_1+H_2}d=1, |\mathbf{k}|=0 \\
c_{17}\,\ln(1+\varepsilon^{-\frac{1}{2}})\varepsilon^{\frac{H_1+H_2}{2H_1H_2}-\frac{d}{2}-|\mathbf{k}|} &   \text{if}\; \frac{H_1H_2}{H_1+H_2}d=1, \frac{H_1H_2}{H_1+H_2}(2|\mathbf{k}|+d)>1\\
c_{17}\,\varepsilon^{\frac{H_1+H_2}{2H_1H_2}-\frac{d}{2}-|\mathbf{k}|} &   \text{if}\; \frac{H_1H_2}{H_1+H_2}d<1,  \frac{H_1H_2}{H_1+H_2}(2|\mathbf{k}|+d)>1\\  
c_{17}\,\ln(1+\varepsilon^{-\frac{1}{2}}) &   \text{if}\;  \frac{H_1H_2}{H_1+H_2}d<1, \frac{H_1H_2}{H_1+H_2}(2|\mathbf{k}|+d)=1. 
\end{array} \right.
\end{align*}
This completes the proof.

\end{proof}

\bigskip

\section{Appendix}

In this section, we give some known results that are used in this paper.

\begin{lemma} \label{a1} For any $T>0$ and $a_i\in(0,1)$ with $i=1,2,\cdots,m$, 
\begin{align*}
\int_{D^m_T} \prod^{m}_{j=1} u_j^{-a_i}\, du=\frac{\prod^m_{j=1}\Gamma(1-a_j)}{\Gamma(m+1-\sum^{m}_{i=1}a_i)}T^{\sum\limits^{m}_{i=1}(1-a_i)},
\end{align*}
where $D^m_{T}=\Big\{0<u_1+u_2+\cdots+u_m<T:\, u_i\geq 0,\, i=1,2,\cdots,m\Big\}$.
\end{lemma}
\begin{proof} For $i=1,2,\cdots, m$, let $d\overline{u}_i=\prod^m_{j=i}du_j$. Then
\begin{align*}
\int_{D^m_{T}}\prod^{m}_{j=1} u_j^{-a_i}\, du
&=\frac{T^{\sum\limits^{m}_{i=1}(1-a_i)}}{1-a_1} \int_{D^{m-1}_1}(1-\sum^m_{j=2} u_j)^{1-a_1} \prod^{m}_{j=2} u_j^{-a_i}\, d\overline{u}_2\\
&=\frac{T^{\sum\limits^{m}_{i=1}(1-a_i)}}{1-a_1} \int_{D^{m-2}_1} \int^{1-\sum\limits^m_{j=3}u_j}_0 u^{-a_2}_2(1-\sum^m_{j=3} u_j-u_2)^{1-a_1}\, du_2 \prod^{m}_{j=3} u_j^{-a_i}  \, d\overline{u}_3\\
&\;\; \vdots\\
&=\frac{T^{\sum\limits^{m}_{i=1}(1-a_i)}}{1-a_1} \prod^m_{j=2}B(1-a_{j}, j-\sum^{j-1}_{i=1}a_i) \\
&=\frac{\prod^m_{j=1}\Gamma(1-a_j)}{\Gamma(\sum^{m}_{i=1}(1-a_i)+1)}T^{\sum\limits^{m}_{i=1}(1-a_i)}.
\end{align*}
\end{proof}

\begin{lemma} \label{a2} For any $T>0$ and $a_i\in(0,1)$ with $i=1,2,\cdots,m$, 
\begin{align*}
\int_{D^m_{T,h}}\prod^{m}_{j=1} u_j^{-a_i}\, du\leq \frac{\prod^m_{j=1}\Gamma(1-a_j)}{\Gamma(m+1-\sum^{m}_{i=1}a_i)}h^{\sum\limits^{m}_{i=1}(1-a_i)},
\end{align*}
where $D^m_{T,h}=\Big\{T<u_1, T<u_1+u_2+\cdots+u_m<T+h:\, u_i\geq 0,\, i=2,\cdots,m\Big\}$.
\end{lemma}
\begin{proof} For $i=1,2,\cdots, m$, let $d\overline{u}_i=\prod^m_{j=i}du_j$. Then
\begin{align*}
\int_{D^m_{T,h}}\prod^{m}_{j=1} u_j^{-a_i}\, du
&=\frac{1}{1-a_1} \int_{D^{m-1}_h}[(T+h-\sum^m_{j=2} u_j)^{1-a_1}-T^{1-a_1}]\, \prod^{m}_{j=2} u_j^{-a_i}\, d\overline{u}_2\\
&\leq \frac{1}{1-a_1} \int_{D^{m-1}_h}(h-\sum^m_{j=2} u_j)^{1-a_1}\, \prod^{m}_{j=2} u_j^{-a_i}\, d\overline{u}_2\\
&=\int_{D^m_h}\prod^{m}_{j=1} u_j^{-a_i}\, du\\
&=\frac{\prod^m_{j=1}\Gamma(1-a_j)}{\Gamma(\sum^{m}_{i=1}(1-a_i)+1)}h^{\sum\limits^{m}_{i=1}(1-a_i)},
\end{align*}
where we used Lemma \ref{a1} in the last equality.
\end{proof}

\begin{lemma} \label{a3}  Assume that $T>0$ and $\alpha\geq 0$. Then for all $\varepsilon\in(0,T/2)$,
\[
\int^T_0 (u^{\frac{2H_1H_2}{H_1+H_2}}+\varepsilon)^{-\frac{d}{2}-\alpha} du\asymp \left\{\begin{array}{cc} \varepsilon^{\frac{H_1+H_2}{2H_1H_2}-\frac{d}{2}-\alpha} & \text{if}\; \frac{H_1+H_2}{H_1H_2}<d+2\alpha \\
\ln(1+\varepsilon^{-\frac{1}{2}}) &   \text{if}\; \frac{H_1+H_2}{H_1H_2}=d+2\alpha \\  
1   &   \text{if}\; \frac{H_1+H_2}{H_1H_2}>d+2\alpha \\  \end{array} \right.
\]
and
\[
\int_{[0,T]^2} \big(u^{2H_1}+v^{2H_2}+\varepsilon\big)^{-\frac{d}{2}-\alpha} du\, dv \asymp \left\{\begin{array}{cc} \varepsilon^{\frac{H_1+H_2}{2H_1H_2}-\frac{d}{2}-\alpha} & \text{if}\; \frac{H_1+H_2}{H_1H_2}<d+2\alpha \\
\ln(1+\varepsilon^{-\frac{1}{2}}) &   \text{if}\; \frac{H_1+H_2}{H_1H_2}=d+2\alpha \\  
1   &   \text{if}\; \frac{H_1+H_2}{H_1H_2}>d+2\alpha \\  \end{array} \right.,
\]
where $f(\varepsilon)\asymp g(\varepsilon)$ means that the ratio $f(\varepsilon)/g(\varepsilon)$ is bounded from below and above by positive constants not depending on $\varepsilon\in(0,T/2)$.
\end{lemma}
\begin{proof}  This first result follows from Lemma 2.2 in \cite{wx}. To get the second one,  we make the change of variables $u^{H_1}=r\cos\theta$ and $v^{H_2}=r\sin\theta$. It is easy to see that  $\int_{[0,T]^2} \big(u^{2H_1}+v^{2H_2}+2\varepsilon\big)^{-\frac{d}{2}-\alpha} du\, dv$ is less than a constant multiple of  $\int^{2T^{H_1}\vee 2T^{H_2}}_0\frac{r^{\frac{H_1+H_2}{H_1H_2}-1}}{ (r^2+\varepsilon)^{\frac{d}{2}+\alpha}} dr$ and greater that a constant multiple of $\int^{T^{H_1}\wedge T^{H_2}}_0\frac{r^{\frac{H_1+H_2}{H_1H_2}-1}}{ (r^2+\varepsilon)^{\frac{d}{2}+\alpha}} dr$. The desired result now follows from Lemma 2.2 in \cite{wx}.
\end{proof}

\begin{lemma} \label{a4} For any $m\in\N$, $x_j\in\R^d$ with $j=1,\cdots, m$ and $x_{m+1}=0$, 
\begin{align*}
\sum^m_{j=1} |x_j-x_{j+1}|^2\geq \frac{2}{m(m+1)} \sum^m_{j=1} |x_j|^2.
\end{align*}
\end{lemma}

\begin{proof}
Make the change of variables $y_j=x_j-x_{j+1}$ for $j=1,\dots,m$. Then 
\begin{align*}
\sum^m_{j=1} |x_j|^2
&=\sum^m_{j=1} |\sum^m_{k=j}y_k|^2\leq \sum^m_{j=1} (m-j+1)\sum^m_{k=j}|y_k|^2\leq \frac{m(m+1)}{2} \sum^m_{k=1}|y_k|^2.
\end{align*}
Substituting the original variables back gives the desired inequality.
\end{proof}

\bigskip

$\begin{array}{cc}
\begin{minipage}[t]{1\textwidth}
{\bf Minhao Hong}\\
School of Statistics, East China Normal University, Shanghai 200262, China \\
\texttt{hongmhecnu@foxmail.com}
\end{minipage}
\hfill
\end{array}$

\medskip

$\begin{array}{cc}
\begin{minipage}[t]{1\textwidth}
{\bf Fangjun Xu}\\
Key Laboratory of Advanced Theory and Application in Statistics and Data Science - MOE, School of Statistics, East China Normal University, Shanghai, 200062, China \\
NYU-ECNU Institute of Mathematical Sciences at NYU Shanghai, 3663 Zhongshan Road North, Shanghai, 200062, China\\
\texttt{fangjunxu@gmail.com, fjxu@finance.ecnu.edu.cn}
\end{minipage}
\hfill
\end{array}$


\bigskip



\begin{thebibliography}{99}

\bibitem{Be} S. M. Berman:  Local nondeterminism and local times of Gaussian processes.
\textit{Indiana Univ. Math.J.}, {\bf 23},  69--94, 1973.


\bibitem{ghx} J. Guo, Y. Hu and Y. Xiao: High-order derivative of intersection local time for two independent fractional Brownian motions. \textit{J. Theor.  Probab.}, in press.

\bibitem{gh}  D. Geman and J. Horowitz: Occupation densities.  \textit{Ann. Probab.},  {\bf 8}, 1--67,  1980.

\bibitem{ghr}  D. Geman and J. Horowitz and J. Rosen: A local time analysis of intersections of Brownian paths in the plane. \textit{Ann. Probab.},  {\bf 12}, 86--107, 1984.



\bibitem{jnp} A. Jaramillo, I. Nourdin and G. Peccati: Approximation of fractional local times zero energy and weak derivatives. arXiv: 1903.08683.

\bibitem{no}  D. Nualart and S. Ortiz-Latorre: Intersection local time for two independent fractional Brownian motions. \textit{J. Theor. Probab.}, 20, 759--757, 2007.


\bibitem{nx2} D. Nualart and F. Xu: Asymptotic behavior for an additive functional of two independent self-similar Gaussian processes. \textit{Stochastic Process. Appl.}, in press.

\bibitem{p} L. Pitt: Local times for Gaussian vector fields. \textit{Indiana Univ. Math.J.}, {\bf 27}(2), 309--330, 1978.

\bibitem{sxy} J. Song, F. Xu and Q. Yu: Limit theorems for functionals of two independent Gaussian processes. \textit{Stochastic Process. Appl.}, in press.


\bibitem{wx}   D. Wu and Y. Xiao: Regularity of intersection local times of fractional Brownian motions.  \textit{J. Theor. Probab.}, {\bf 23}(4), 972--1001, 2010.


\bibitem{y} L. Yan: Derivative for the intersection local time of fractional Brownian motions. arXiv:1403.4102.
\end{thebibliography}
\end{document}